\documentclass[11pt, a4paper]{article}

\usepackage{latexsym}
\usepackage{graphicx, float}
\usepackage{amssymb,amsfonts,amsthm,amsmath,graphicx,url}
\usepackage{enumerate}

\usepackage{color}
\usepackage{pifont}
\usepackage{amsfonts}
\usepackage{curves}
\usepackage{epsfig}
\usepackage{mathrsfs}

\input amssym.def
\input amssym.tex

\definecolor{VeryLightBlue}{rgb}{0.9,0.9,1}
\definecolor{LightBlue}{rgb}{0.8,0.8,1}
\definecolor{MidBlue}{rgb}{0.5,0.5,1}
\definecolor{DarkBlue}{rgb}{0,0,0.6}
\definecolor{Blue}{rgb}{0,0,1}
\definecolor{Gold}{rgb}{1,0.843,0}
\definecolor{LightGreen}{rgb}{0.88,1,0.88}
\definecolor{MidGreen}{rgb}{0.6,1,0.6}
\definecolor{DarkGreen}{rgb}{0,0.6,0}
\definecolor{VeryLightYellow}{rgb}{1,1,0.9}
\definecolor{LightYellow}{rgb}{1,1,0.6}
\definecolor{MidYellow}{rgb}{1,1,0.5}
\definecolor{DarkYellow}{rgb}{1,1,0.2}
\definecolor{DarkPurple}{rgb}{.6,0,1}
\definecolor{Red}{rgb}{1,0,0}
\definecolor{VeryLightRed}{rgb}{1,0.9,0.9}
\definecolor{LightRed}{rgb}{1,0.8,0.8}
\definecolor{MidRed}{rgb}{1,0.55,0.55}

\long\def\delete#1{}

\newcommand{\n}{\noindent}

\newtheorem{thm}{Theorem}[section]
\newtheorem{conj}[thm]{Conjecture}
\newtheorem{lemma}[thm]{Lemma}

\newtheorem{defn}[thm]{Definition}

\newcommand{\be}{\begin{equation}}
\newcommand{\ee}{\end{equation}}
\newcommand{\bea}{\begin{eqnarray}}
\newcommand{\eea}{\end{eqnarray}}
\newcommand{\bean}{\begin{eqnarray*}}
\newcommand{\eean}{\end{eqnarray*}}

\def\qed{\hfill$\Box$\vspace{11pt}}

\newcommand{\RRR}{\mathcal{R}}

\begin{document}

\title{Zero-sum 6-flows in 5-regular graphs}
\author{Fan Yang\thanks{Department of Mathematics and Physics, Jiangsu University of Science and Technology, Zhenjiang, Jiangsu 212003, China. Supported by the NSF-China (11326215, 11501256 and 11371009), Corresponding author email: fanyang$\_$just@163.com}\;\,,  Xiangwen Li\thanks{Department of Mathematics, Huazhong Normal
University, Wuhan 430079, China. Partially supported by NSF-China Grant (11171129) and by
Doctoral Fund of Ministry of Education of China (20130144110001). }
}

\date{26/01/2016}
\openup 0.5\jot
\maketitle

\begin{abstract}
Let $G$ be a graph. A zero-sum flow of $G$ is an assignment of non-zero real
numbers to the edges of $G$ such that the sum of the values of all edges incident with
each vertex is zero. Let $k$ be a natural number. A zero-sum $k$-flow is a flow with
values from the set $\{\pm1, \ldots, \pm(k - 1)\}$. In this paper, we prove that every 5-regular
graph admits a zero-sum 6-flow.
\end{abstract}

\section{Introduction}
\label{sec:int}

All graphs in this paper are finite and undirected without loops, possibly with parallel edges. A \emph{nowhere-zero $k$-flow} in a graph with orientation is an assignment of an integer from $\{-(k-1), \ldots, -1, 1, \ldots, (k-1)\}$ to each of its edges such that Kirchhoff's law is respected, that is, the total incoming flow is equal to the total outgoing flow at each vertex. As noted in \cite{Jaeger}, the existence of a nowhere-zero flow of a
graph $G$ is independent of the choice of the orientation. Nowhere-zero flows in graphs were introduced by Tutte\cite{Tutte} in 1949. A great deal of research in the area has
been motivated by Tutte's 5-Flow Conjecture which asserts that every 2-edge-connected
graph admits a nowhere-zero 5-flow. In 1983, Bouchet\cite{Bouchet} generalized
this concept to bidirected graphs. A {\em bidirected graph} $G$ is a graph with vertex set $V(G)$
and edge set $E(G)$ such that each edge is oriented as one of the four possibilities in Figure 1.

\begin{figure}[ht]
\centering
\includegraphics*[height=2.5cm]{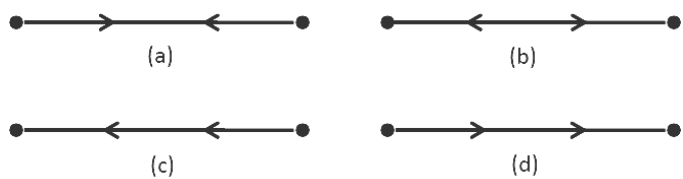}
\caption{\small Orientations of edges in bidirected graph.}
\label{fig:1}
\end{figure}

\n An edge with orientation as (a) (respectively, (b)) is called an {\em in-edge} (respectively, {\em out-edge}). An edge that is neither an in-edge nor an out-edge is called an {\em ordinary edge} as in (c) or (d). An integer-valued function $f$ on $E(G)$ is a nowhere-zero bidirected $k$-flow if for every $e\in E(G)$, $0<|f(e)|<k$, and at every vertex $v$, the value of function $f$
in equals the value of function $f$ out.
The following conjecture, posed by Bouchet \cite{Bouchet} and known as Bouchet's $6$-flow conjecture, is one of the most important problems on nowhere-zero integer flows in bidirected graphs.

\begin{conj}
\label{conj:Bou}
(\cite{Bouchet}) If a bidirected graph admits a nowhere-zero $k$-flow for some positive integer $k$, then it admits a nowhere-zero 6-flow.
\end{conj}

In \cite{Bouchet}, Bouchet showed that the value $6$ in this conjecture is best possible. Raspaud and Zhu \cite{Raspaud} proved that, if a 4-edge-connected bidirected graph admits a nowhere-zero integer flow, then it admits a nowhere-zero 4-flow, which is best possible. Recently, DeVos \cite{DeVos} proved that the result in Bouchet's Conjecture is true if 6 is replaced by 12.

A \emph{zero-sum flow} (also called \emph{nowhere-zero unoriented flow}) in a graph is an assignment of non-zero integers to the edges of $G$ such
that the total sum of the assignments of all edges incident with any vertex of $G$ is zero. A zero-sum
$k$-flow in a graph $G$ is a zero-sum flow with flow values from the set $\{\pm 1, \ldots, \pm(k-1)\}$.
The following conjecture is known as \textbf{Zero-Sum Conjecture}.

\begin{conj}
\label{conj:ZSC}
(\cite{Akbari0}) If a graph $G$ admits a zero-sum flow, then it admits
a zero-sum 6-flow.
\end{conj}

Indeed a
zero-sum $k$-flow in $G$ is exactly a nowhere-zero $k$-flow in the bidirected graph with underlying graph $G$ such that each edge is an in-edge or out-edge. The following theorem shows the equivalence between Bouchet's Conjecture and Zero-Sum Conjecture.

\begin{thm}
\label{equal}
(\cite{Akbari1})
Bouchet's Conjecture and Zero-Sum Conjecture are equivalent.
\end{thm}

There are many results about Zero-Sum Conjecture recently. S. Akbari, A. Daemi, et al. \cite{Akbari2} proved that Zero-Sum Conjecture is true for hamiltonian graphs, with 6 replaced by 12.
The following conjecture is a stronger version of Zero-Sum Conjecture for regular graphs.

\begin{conj}
\label{regular conj}
(\cite{Akbari1})
Every $r$-regular graph ($r\ge 3$) admits a zero-sum 5-flow.
\end{conj}

Motivated by Conjecture \ref{regular conj}. Akbari et al. \cite{Akbari1} proved that every $r$-regular graph ($r\ge 3$) admits a zero-sum 7-flow.
Moreover, Akbari et al. showed that Conjecture \ref{regular conj} is true except for 5-regular graph.

\begin{thm}
\label{regular}
(\cite{Akbari})
Every $r$-regular
graph, where $r\ge 3$ and $r\neq 5$, admits a zero-sum 5-flow.
\end{thm}

In this paper we prove the existence of zero-sum 6-flow in 5-regular graph.

\begin{thm}
\label{th1}
Every 5-regular graph admits a zero-sum 6-flow.
\end{thm}

We will prove this result in Section 3. In Section 2, we introduce a new graph, called cycle-cubic tree, and give a proposition of cycle-cubic tree, which is useful to construct a zero-sum 6-flow in 5-regular graph.

\section{Preliminaries}
\label{sec:pre}

A \emph{cycle} in this paper is meant a connected 2-regular graph. Moreover, an even (or odd) cycle is a cycle with length even (or odd). Two parallel edges form a cycle of length two.
A \emph{cubic} graph is a connected 3-regular graph. Let $P$ be a path and $x, y\in V(P)$, denote by $P[x, y]$ the subpath of $P$ between $x$ and $y$, and $|P|$ the length of $P$.
In this paper, a vertex $v$ can be treated as a trivial path with length $0$.
Let $P_1$ and $P_2$ are two vertex-disjoint paths of $G$, \emph{a path connecting $P_1$ and $P_2$} is a path with all internal vertices outside $P_1\cup P_2$ between one vertex of $P_1$ and one vertex of $P_2$.
A \emph{$k$-factor} of a graph is a $k$-regular spanning subgraph.
In particular, a 2-factor is a disjoint union of cycles that cover all vertices of the graph.
For integers $a$ and $b$, $1\le a\le b$, an \emph{$[a, b]$-factor} of
$G$ is defined to be a spanning subgraph $F$ of $G$ such that $a\le d_F(v)\le b$ for every $v\in V(G)$, where $d_F(v)$ is the degree of $v$ in $F$. For any
vertex $v\in V(G)$, let $N_G(v)=\{u\in V(G)| uv\in E(G) \}$.

\begin{lemma}
\label{factor}
(\cite{Kano})
Let $r\ge 3$ be an odd integer and $k$ an integer such that $1\le k \le 2r/3$.
Then every $r$-regular graph has a $[k-1, k]$-factor each component of which is regular.
\end{lemma}

\delete{
\begin{lemma}
\label{factor sum}
(\cite{Akbari})
Let G be an $r$-regular graph. Then for every even integer $q$, $2r\le q\le 4r$, there
exists a function $f: E(G)\rightarrow \{2, 3, 4\}$ such that for every $u\in V(G)$,
$\Sigma_{v\in N_G(u)}f(uv)=q$.
\end{lemma}
}

\delete{The following result can be easily verified.

\begin{lemma}
\label{even}
(\cite{Akbari0})
Let $r$ be an even integer with $r\ge 4$. Then every $r$-regular graph has a
zero-sum 3-flow.
\end{lemma}

\begin{lemma}
\label{r}
(\cite{Akbari1})
Let $G$ be an $r$-regular graph. If $r$ is divisible by 3, then $G$ has a zero-sum
5-flow.
\end{lemma}
}

Next we give a basic lemma of connected graph which will be used in the proof of Lemma \ref{c-c tree}.

\begin{lemma}
\label{m path}
Let $G$ be a connected graph. Then for any even number of distinct vertices of $G$, say $u_1, u_2, \ldots, u_{2m}$ ($m\ge 1$),
there exist $m$ edge-disjoint paths in $G$ such that the $2m$ end-vertices of these $m$ paths are pairwise distinct and form the set $\{u_1, u_2, \ldots, u_{2m}\}$.
\end{lemma}
\begin{proof}
Since $G$ is connected, the result is trivial when $m=1$.

\delete{Suppose $m=2$. Since $G$ is a cubic graph, there exists a path $P$ between $u_1$ and $u_2$. If there exists a path $P'$ between $u_3$ and $u_4$ such that $P$ and $P'$ are edge-disjoint, then let $P_1=P, P_2=P'$, we are done.
Thus each $P$ and $P'$ have common edges, so have common vertices. Let $x$ (and $y$) be the first (and the last) common vertex when traveling $P'$ from $u_3$ to $u_4$.
Without loss of generality, we assume that the length of $P[u_1, x]$ is less than the length of $P[u_1, y]$. In this case, let $P_1$ be the concatenation of path $P[u_1, x]$ and $P'[x, u_3]$; $P_2$ be the concatenation of path $P[u_2, y]$ and $P'[y, u_4]$. Clearly, $P_1$ and $P_2$ are edge-disjoint paths with different end-vertices in $\{u_1, u_2, u_3, u_{4}\}$.}

Assume that the result is true when the even number of vertices is less than $2m$ for $m \ge 2$.
Now consider $2m$ distinct vertices $u_1, u_2, \ldots, u_{2m}$ of $G$. By induction hypothesis, $G$ contains $m-1$ edge-disjoint paths $P^1, \ldots, P^{m-1}$ such that the $2m-2$ end-vertices of them are pairwise distinct and form the set $\{u_1, u_2, \ldots, u_{2m-2}\}$.
Without loss of generality, we assume that $P^i$ is a path between $u_{2i-1}$ and $u_{2i}$ for $1\le i\le m-1$.
If there exists a path $P$ in $G$ between $u_{2m-1}$ and $u_{2m}$ such that $P$ and $P^i$ are edge-disjoint for each $1\le i\le m-1$, then $P^1, \ldots, P^{m-1}, P$ are desired edge-disjoint paths.
Thus we assume that each path of $G$ between $u_{2m-1}$ and $u_{2m}$ has common edges with some paths in $\{P^1, P^2, \ldots, P^{m-1}\}$, so has common vertices with such paths.
Choose a path of $G$ between $u_{2m-1}$ and $u_{2m}$, say $P$.
Let $x, y\in V(P)$, possibly with $x=y$, be the first and last common vertices with $\cup_{i=1}^{m-1}{P^i}$ when traveling along $P$ from $u_{2m-1}$ to $u_{2m}$, respectively.
If $x, y$ are vertices of the same $P^i$, say $x, y\in V(P^1)$, then without loss of generality, we assume that $|P^1[u_1, x]|\le |P^1[u_1, y]|$. In this case, let $P_0$ be the concatenation of $P^1[u_1, x]$ and $P[x, u_{2m-1}]$, and $P_1$ the concatenation of $P^1[u_2, y]$ and $P[y, u_{2m}]$. Let $P_{j}=P^j$ for $2\le j\le m-1$. Clearly, $P_0, P_1, \ldots, P_{m-1}$ are desired edge-disjoint paths.

Thus we assume that $x$ and $y$ are in different $P^i$ for $i\in \{1, 2, \ldots, m-1\}$.
Without loss of generality, we assume that $x\in V(P^1)$ and $y\in V(P^2)$. If there exists a subpath of $P$ connecting $P^1$ and $P^2$ with all internal vertices outside $\cup_{i=3}^{m-1}{P^i}$, then without loss of generality, we assume that $P[x_1, y_1]$ is such a subpath with $x_1\in V(P^1)$, $y_1\in V(P^2)$, $|P^1[u_1, x]|\le |P^1[u_1, x_1]|$ and $|P^2[u_{3}, y_1]|\le |P^2[u_{3}, y]|$. In this case, let $P_1$ be the concatenation of $P^1[u_1, x]$ and $P[x, u_{2m-1}]$, $P_2$ the concatenation of $P^2[u_{4}, y]$ and $P[y, u_{2m}]$, and $P_m$ the concatenation of $P^1[x_1, u_{2}]$, $P[x_1, y_1]$ and $P^2[u_{3}, y_1]$. Let $P_{j}=P^j$ for $3\le j\le m-1$. Then $P_1, P_2, \ldots, P_{m}$ are desired edge-disjoint paths.
Then there exists no subpath of $P$ connecting $P^1$ and $P^2$ with all internal vertices outside $\cup_{i=3}^{m-1}{P^i}$.
This means that any subpath of $P$ connecting $P^1$ and $P^2$ has common vertices with $\cup_{i=3}^{m-1}{P^i}$. Thus $P$ can be decomposed into the union of subpaths of $\{P^1, P^2, \ldots, P^{m-1}\}$ and subpaths  connecting any two paths in $\{P^1, P^2, \ldots, P^{m-1}\}$ each with all internal vertices outside $\cup_{i=1}^{m-1}{P^i}$.
Without loss of generality, we can relabel the index of paths in $\{P^3, P^4, \ldots, P^{m-1}\}$ such that
there exists a subpath of $P$ connecting $P^1$ and $P^3$, a subpath of $P$ connecting $P^t$ and $P^2$ (where $t$ is an integer in $\{3, \ldots, m-2\}$), a subpath of $P$ conneting $P^j$ and $P^{j+1}$ (for each $3\le j\le t-1$)  each with all internal vertices outside $\cup_{i=1}^{m-1}{P^i}$  when traveling along $P$ from $u_{2m-1}$ to $u_{2m}$.
Let $P[x_j, y_j]$ be such a subpath of $P$ connecting $P^j$ and $P^{j+1}$ such that $x_j\in V(P^j)$ and $y_j\in V(P^{j+1})$ for each $3\le j\le t-1$, $P[x_1, y_1]$ be such a subpath of $P$ connecting $P^1$ and $P^{3}$ such that $x_1\in V(P^1)$ and $y_1\in V(P^{3})$, $P[x_2, y_2]$ be such a subpath of $P$ connecting $P^t$ and $P^{2}$ such that $x_2\in V(P^t)$ and $y_2\in V(P^{2})$.
Without loss of generality, we assume that $|P^j[u_{2j-1}, x_j]|\le |P^j[u_{2j-1}, y_{j-1}]|$ for each $4\le j\le t-1$, $|P^1[u_{1}, x]|\le |P^1[u_{1}, x_{1}]|$, $|P^2[u_{3}, y_{2}]|\le |P^2[u_{3}, y]|$ and $|P^3[u_{5}, x_3]|\le |P^3[u_{5}, y_{1}]|$ . In this case, let $P_0$ be the concatenation of $P^1[u_1, x]$ and $P[x, u_{2m-1}]$;  $P_1$ the concatenation of $P^{1}[u_{2}, x_{1}]$, $P[x_{1}, y_{1}]$ and $P^{3}[u_{6}, y_{1}]$; $P_2$ the concatenation of $P^{2}[u_{4}, y]$ and $P[y, u_{2m}]$; $P_{t}$ the concatenation of $P^{t}[u_{2t-1}, x_{2}]$, $P[x_{2}, y_{2}]$ and $P^{2}[u_{3}, y_{2}]$;
$P_j$ the concatenation of $P^{j}[u_{2j-1}, x_{j}]$, $P[x_{j}, y_{j}]$ and $P^{j+1}[y_{j}, u_{2j+2}]$ for $3\le j\le t-1$.
Let $P_{j}=P^j$ for $t+1\le j\le m-1$.
Then $P_0, P_1, \ldots, P_{m-1}, P_m$ are desired edge-disjoint paths.\qed
\end{proof}

%As usual, for a vertex or edge set $X$ of $G$, denote by $G[X]$ the subgraph of $G$ induced by $X$.
For disjoint subsets $X, Y$ of $V(G)$, denote by $E_G(X, Y)$ the set of edges of $G$ with one vertex in $X$ and the other in $Y$.
If $H$ and $K$ are vertex-disjoint subgraphs of $G$, then we use $E_G(H, K)$ in place of $E_G(V(H), V(K))$.
Given a subgraph $H$ of $G$, we use $G/H$ to denote the graph obtained from $G$ by contracting all edges in $H$ and deleting the resultant loops.

\begin{defn}
\label{def:c-c tree}
{\em
A graph $G$ is called a \emph{cycle-cubic tree} relative to a family $\RRR = \{R_1, R_2, \ldots, R_t\}$, the size of this family $\RRR$ is $t$, of pairwise vertex-disjoint connected graphs $R_1, R_2, \ldots, R_t$ of $G$ if the following conditions are satisfied:
\begin{enumerate}[{(a)}]

\item $R_i$ is a cycle or a cubic graph and $V(R_1)\cup \ldots \cup V(R_t)=V(G)$;

\item the graph $G_{\RRR}$ ($=G/(R_1\cup R_2\cup \ldots \cup R_t)$) obtained from $G$ by identifying all vertices in each $R_i$ to form a single vertex $v_{R_i}$, $1\le i\le t$, is a tree;

\item the degree of $v_{R_i}$ in $G_{\RRR}$ is even (respectively, odd) if and only if $R_i$ is an even cycle or a cubic graph (respectively, an odd cycle).

%\item apart from those $v_{R_i}$ from odd cycle $R_i$, $G_{\RRR}$ has no other odd-degree vertices.
\end{enumerate}

}
\end{defn}

By Definition~\ref{def:c-c tree}, $d_G(v)\ge 2$ for each vertex $v\in V(G)$, and if $t=1$, then $G$ is an even cycle or a cubic graph.
Next we prove the existence of a special function for cycle-cubic trees which is crucial to the study of zero-sum 6-flows in 5-regular graphs.

\begin{lemma}
\label{c-c tree}
Every cycle-cubic tree $G$ relative to a family $\RRR = \{R_1, R_2, \ldots, R_{t}\}$ with maximum degree at most 5 admits a function $f: E(G) \rightarrow Z$ such that
\be
\label{eq:fe}
f(e)=\left\{
\begin{array}{ll}
\text{$4$ or $5$}, &\text{if $e\in E(R_i)$ and $R_i$ is an edge of a cycle},\\ [0.1cm]
\text{$1, 2$ or $3$}, &\text{if $e\in E(R_i)$ and $R_i$ is an edge of a cubic graph},\\ [0.1cm]
\text{$-2$ or $-4$}, &\text{if $e$ is an edge outside $\cup_{i=1}^{t}R_i$}.
\end{array}
\right.
\ee
and
for each vertex $u\in V(G)$,

\be
\label{eq:ff}
\Sigma_{v\in N_G(u)}f(uv) = \left\{
\begin{array}{ll}
0, &\text{if $d_G(u)=5$},\\ [0.1cm]
3, &\text{if $d_G(u)=4$},\\ [0.1cm]
6, &\text{if $d_G(u)=3$},\\ [0.1cm]
9, &\text{if $d_G(u)=2$}.
\end{array}
\right.
\ee

\end{lemma}
\begin{proof}
By Definition \ref{def:c-c tree} (b), $G$ is connected.
%$d_{G_{\RRR}}(v_{R_i})$ is odd if and only if $R_i$ is an odd cycle.
%This means the number of odd cycle in $\RRR$ is even.
If $t=1$, then $G\cong R_1$ and $R_1$ is an even cycle or a cubic graph by Definition \ref{def:c-c tree} (c).
If $R_1$ is an even cycle, then define $f: E(G) \rightarrow \{4, 5\}$ such that $4, 5, \ldots, 4, 5$ occur alternatively on the cycle with a chosen orientation. Clearly, $f$ is a function satisfying (1) and (2).
If $R_1$ is a cubic graph, then define $f: E(G) \rightarrow \{2\}$. Clearly, $f$ also satisfies (1) and (2).

\delete{Suppose $t=2$, that is, $\RRR=\{R_1, R_2\}$.
In this case, $R_1$ and $R_2$ are odd cycles. Then $G$ is a graph that is consisted of $R_1$ and $R_2$ and an edge outside $E(\RRR)$ connecting this two bad cycles.
Without loss of generality, we assume that $u_1u_2\notin E(\RRR)$ with $u_i\in V(R_i)$ for each $i=1, 2$. We label edges of $R_i$ starting with an edge incident with $u_i$ by $4, 5, \ldots, 4, 5, 4$ alternatively in a direction and then label $u_1u_2$ with $-2$. Thus these labels constitute a desired mapping on $E(G)$ that meets (1) and (2).}

Assume that the result is true when the size of $\RRR$ less than $t$ for some $t\ge 2$.
Next consider a cycle-cubic tree $G$ relative to $\RRR=\{R_1, R_2, \ldots, R_t\}$. By Definition \ref{def:c-c tree}, $\RRR$ contains at least two odd cycles.
Since $G/{\RRR}$ is a tree, we may choose an odd cycle, say $R_1$, such that $d_{G_{\RRR}}(v_{R_1})=1$. We assume that the neighbor of $v_{R_1}$ in $G_{\RRR}$ is $v_{R_2}$ and $v_1u_1\in E_G(R_1, R_2)$, where $v_1\in V(R_1)$ and $u_1\in V(R_2)$. Since the maximum degree of $G$ is at most 5, $1 \le |E_{G}(u_1, V(G)-V(R_2))|\le 3$.

Case 1. $|E_{G}(u_1, V(G)-V(R_2))|= 1$.

Subcase 1.1. $R_2$ is a cycle.

In this case, $d_G(u_1)=3$ and $u_1$ has two neighbors in $R_2$, say $u_{11}$ and $u_{12}$. Deleting $V(R_1)\cup \{u_1\}$, and then adding the edge $u_{11}u_{12}$ from $G$, we obtain a cycle-cubic tree $G'$ relative to $\RRR'=\{R_2', R_3 \ldots, R_t\}$, where $R_2'$ is a cycle with $|R_2'|=|R_2|-1$.
By the induction hypothesis, $G'$ has a function $f'$ satisfying (1) and (2).

Let $f'(u_{11}u_{12})=a$. By (1), $a\in \{4, 5\}$. Define $f(u_1u_{11})=f(u_1u_{12})=a$ and $f(u_1v_1)=6-2a\in \{-2, -4\}$. Then $f(v_1u_1)+f(u_1u_{11})+f(u_1u_{12})=a+a+(6-2a)=6$, $u_1$ satisfies (2).
If $a=4$, then label the edges of $R_1$ with $4$ and $5$ alternatively in a chosen direction starting with an edge incident with $v_1$.
If $a=5$, then label the edges of $R_1$ with $5$ and $4$ alternatively in a chosen direction starting with an edge incident with $v_1$.
In either case, let $f(e)=f'(e)$ for all other edges of $G$. We can deduce that $f$ is a function of $G$ satisfying (1) and (2).

Subcase 1.2. $R_2$ is a cubic graph.

By Definition \ref{def:c-c tree} (c), $|E_{G}(R_2, G-V(R_2))|$ is even. Since the maximum degree of $G$ is at most 5, $0 \le |E_{G}(u, V(G)-V(R_2))|\le 2$ for each vertex $u\in V(R_2)$. Then the number of vertex $u$ of $R_2$ with $|E_{G}(u, V(G)-V(R_2))|=1$ is even. Without loss of generality, we may assume that $u_1, u_2, \ldots, u_{2m}$ are all vertices with $|E_{G}(u_i, V(G)-V(R_2))|=1$ for $1\le i\le 2m$ and if there exists vertex $u$ such that $|E_{G}(u, V(G)-V(R_2))|=2$, then let $u_{2m+1}, \ldots, u_{2m+k}$ are all vertices with $|E_{G}(u_{2m+i}, V(G)-V(R_2))|=2$ for $1\le i\le k$.
By Lemma \ref{m path}, we can find $m$ edge-disjoint paths in $R_2$ such that the $2m$ end-vertices of them are pairwise distinct and constitute the set $\{u_1, u_2, \ldots, u_{2m}\}$. Without loss of generality, we may assume that $P_{1}$, $P_{2}$, $\ldots$, $P_{m}$ are such paths in $R_2$, where $P_{i}$ connects $u_{2i-1}$ and $u_{2i}$. Let $P_{k+i}: u_{2m+i}$ be the trivial path in $R_2$ for $1\le i\le k$. Then $P_1, \ldots, P_{m+k}$ are edge-disjoint paths in $R_2$.
Next we construct a new graph $G'$ from $G$ by the following steps:

\emph{Step 1.} Delete all vertices of $R_2$ except vertices $\{u_{1}, \ldots, u_{2m}, u_{2m+1}, \ldots, u_{2m+k}\}$ and delete all edges in $R_2$;

\emph{Step 2.}  Add two parallel edges $u_{2i-1}u_{2i}$ if $|P_i|$ is odd; add two 2-paths $u_{2i-1}x_iu_{2i}$ and $u_{2i-1}y_iu_{2i}$ if $|P_i|$ is even for $1\le i\le m$;

\emph{Step 3.} Split each $u_{2m+i}$ into two vertices $w_{2m+i}$ and $z_{2m+i}$ for $1\le i\le k$ such that one edge of $E_G(u_{2m+i}, V(G)-V(R_2)))$ is incident with $w_{2m+i}$, the other edge of $E_G(u_{2m+i}, V(G)-V(R_2)))$ is incident with $z_{2m+i}$;

\emph{Step 4.}  Add two 2-paths $w_{2m+i} x^{2m+i} z_{2m+i}$ and $w_{2m+i} y^{2m+i} z_{2m+i}$ for $1\le i\le k$.

By Steps 1-4, we get a graph $G'$ which is a union of cycle-cubic trees each with smaller number of cycles and cubic graphs. By the induction hypothesis, there is a function $f'$ of $G'$ satisfying (1) and (2).
Next we observe values under $f'$ of edges which are corresponding edges of $E_{G}(R_2, G-V(R_2))$.

Without loss of generality, we assume that $u_iv_i\in E_{G}(R_2, G-V(R_2))$ for $1\le i\le 2m$ and $u_{2m+j}v_{2m+j}, u_{2m+j}v_{2m+j}^*\in E_{G}(R_2, G-V(R_2))$ for $1\le j\le k$. Without loss of generality, we assume that $w_{2m+j}v_{2m+j}$ and $z_{2m+j}v_{2m+j}^*$ are edges of $G'$ for each $1\le j\le k$. By (1), $f'(u_iv_i), f'(w_jv_{2m+j})$ and $f'(z_jv_{2m+j}^*) \in \{-4, -2\}$ for $1\le i\le 2m$, $1\le j\le k$. If $|P_i|$ is odd, then $P_i$ must be path in $\{P_1, P_2, \ldots, P_m\}$ and the two parallel edges $u_{2i-1}u_{2i}$ with equal values $a$ under $f'$, where $a\in \{4, 5\}$ by (1), (2). Then in this case $f'(u_{2i-1}v_{2i-1})=f'(u_{2i}v_{2i})$. If $|P_i|$ is even, then when $i\in \{1, 2, \ldots, m\}$, $f'(u_{2i-1}x_i)=f'(u_{2i-1}y_i)=a$ and $f'(u_{2i}x_i)=f'(u_{2i}y_i)=b$, and when $i\in \{2m+1, \ldots, 2m+k\}$, $f'(w_{i}x^i)=f'(w_{i}y^i)=a$ and $f'(z_{i}x^i)=f'(z_{i}y^i)=b$, where $\{a, b\}=\{4, 5\}$ by (1), (2). Then in this case, $f'(u_{2i-1}v_{2i-1})\neq f'(u_{2i}v_{2i})$ for $1\le i\le m$ and $f'(w_{i}v_{i})\neq f'(z_iv_{i}^*)$ for $2m+1\le i\le 2m+k$.

Next we define a function of $G$ based on $f'$. If $f'(u_{2i-1}v_{2i-1})=-2$ (or $-4$) for each $1\le i\le m$, then we label $E(P_{i})$ with $1, 3, 1, 3, \ldots$ (or $3, 1, 3, 1,\ldots$) alternatively
starting with the edge incident with $u_{2i-1}$.
%If the length of $P_{i}$ is odd and $f'(u_{2i-1}v_{2i-1})=-2$ (or $-4$), then we label $E(P_{i})$ with $1, 3, \ldots,  1, 3, 1$ (or $3, 1, \ldots, 3, 1, 3$), alternatively.
Then let  $f(e)=2$ for each $e\in E(R_2)-\cup_{i=1}^{m}E(P_i)$, $f(u_{2m+i}v_{2m+i})=f'(w_{2m+i}v_{2m+i})$, $f(u_{2m+i}v_{2m+i}^*)=f'(z_{2m+i}v_{2m+i}^*)$ and $f(e)=f'(e)$ for all other edges $e\in E(G)-E(R_2)$.
Then we can deduce that $f$ is a desired function of $G$ satisfying (1) and (2) since $P_1, P_2, \ldots, P_{m}$ are pairwise edge-disjoint paths.

Case 2. $|E_{G}(u_1, V(G)-V(R_2))|\ge 2$.

Without loss of generality, we may assume that $u_1v_1, u_1v_2\in E_{G}(u_1, V(G)-V(R_2))$ with $v_1\in V(R_1)$. Then $v_2$ is a vertex of $R_i$ for some $i\neq 1, 2$.
We know that $u_1$ is a vertex-cut of $G$. Let $G_1$ and $G_2$ be edge-disjoint subgraphs of $G$ such that $G_1\cup G_2=G$ and $V(G_1)\cap V(G_2)=\{u_1\}$, $u_1v_1, u_1v_2\in E(G_1)$ and the other edges incident with $u_1$ are in $G_2$.
Constructing a graph $G_1'$ from $G_1$ by replacing $u_1$ with a 4-cycle $xyzwx$ and replacing $u_1v_1, u_1v_2$ by $v_1x, v_2z$, respectively.
Clearly, $G_1'$ and $G_2$ are cycle-cubic tree each relative to a smaller number of cycles and cubic graphs.
By the induction hypothesis, $G_1'$ admits a desired function $f_1$ on $E(G_1')$ and $G_2$ admits a desired function $f_2$ on $E(G_2)$, both satisfying (1) and (2).
By the similar discussion in Subcase 1.2, $f_1(v_1x)\neq f_1(v_2z)\in \{-2, -4\}$.
Without loss of generality, we may assume that $f_1(v_1x)=-2$ and $f_1(v_2z)=-4$ by (1) and (2). In this case, we define $f$ on $G$ such that $f(u_1v_1)=f_1(v_1x)$, $f(u_1v_2)=f_1(v_2z)$, and for all other edges $e$ of $G$, let $f(e)=f_i(e)$ if $e\in E(G_i)$ for $i=1, 2$. Clearly, $f$ satisfies (1). We only need to prove that for vertex $u_1$, (2) satisfies. Obviously, $d_{G}(u_1)=d_{G_2}(u_1)+2$ and $\Sigma_{v\in N_{G}(u_1)}f(u_1v)=\Sigma_{v\in N_{G_2}(u_1)}f_2(u_1v)+(-2-4)=\Sigma_{v\in N_{G_2}(u_1)}f_2(u_1v)-6$. Then when $d_{G_2}(u_1)=2$ (or $3$), then $d_{G}(u_1)=4$ (or $5$) and $\Sigma_{v\in N_{G}(u_1)}f(u_1v)=9-6=3$ (or $=6-6=0$), which satisfies (2). Thus $f$ is a desired function on $G$ satisfying (1) and (2).
\qed

\end{proof}

\section{Proof of Lemma \ref{th1}}
\label{sec:pf}

Let $G$ be a graph and $T$ a subset of $V(G)$ with even cardinality. A set $J \subseteq E(G)$ is called a \emph{$T$-join} \cite{Korte} in $G$ if for any $v\in V(G)$, $v$ is incident with an odd number of edges in $J$ if and only if $v\in T$. We also call the subgraph induced by $J$ a $T$-join.

\begin{lemma}
\label{T-join}
(\cite [Proposition 12.6] {Korte})
Let $G$ be a graph and $T\subseteq V(G)$ with $|T|$ even. There exists a $T$-join in $G$
if and only if $|V(Q)\cap T|$ is even for each connected component $Q$ of $G$.
\end{lemma}

\smallskip

\begin{proof}\textbf{of Theorem~\ref{th1}}~
Let $G$ be a 5-regular graph. Without loss of generality, we may assume that $G$ is connected. If $G$ contains a 2-factor, say $G_1$,  then $G_2=G-E(G_1)$ is a 3-factor of $G$. Define $f(e)=3$ for each $e\in E(G_1)$ and $f(e)=-2$ for each $e\in E(G_2)$. Thus $f$ is a zero-sum 4-flow of $G$, so is a zero-sum 6-flow of $G$.

Now we assume that $G$ does not contain a 2-factor or a 3-factor.
By Lemma \ref{factor}, $G$ has a $[2, 3]$-factor each component of which is regular. We may assume that $H=H_1\cup H_2$ is the $[2, 3]$-factor of $G$ such that $H_1$ is 2-regular and $H_2$ is 3-regular.
Clearly, $H_1$ is a union of cycles.
If $H_1$ contains no odd cycles, then we can label the edges of $H_1$ by $4$ and $5$, alternatively. This labeling defines a function $f_1$ on $E(H_1)$. Then for every $u\in V(H_1)$, $\Sigma_{v\in N_{H_1}(u)}f_1(uv)=9$.
Define $f_2(e)=2$ for each $e\in E(H_2)$. Then $f_2$ is a function on $E(H_2)$ such that for every $u\in V(H_2)$, $\Sigma_{v\in N_{H_2}(u)}f_2(uv)=6$.
Define $f(e)=f_1(e)$ for each $e\in E(H_1)$; $f(e)=f_2(e)$ for each $e\in E(H_2)$, and $f(e)=-3$ for each $e\in E(G)\backslash E(H)$. Then $f$ is a zero-sum 6-flow of $G$.

It remains to consider the case where $H_1$ contains odd cycles. Since $G$ is 5-regular and $H_2$ is 3-regular, $|V(G)|$ and $|V(H_2)|$ are even. Thus the number of odd cycles of $H_1$ is even.
Let $\{C_1, C_2, \ldots, C_{2m}\}$ be the family of odd cycles in $H_1$, where $m\ge 1$.
%Let $G/H$ be the graph obtained from $G$ by contracting each connected component of $H$ to a vertex and deleting the resulting loops.
Since $G$ is connected, so is $G/H$. Since $H$ is a spanning subgraph of $G$, each vertex of $G/H$ is a contracted vertex.
Let $T=\{v_{C_1}, v_{C_2}, \ldots, v_{C_{2m}}\}$. By Lemma \ref{T-join}, there exists a $T$-join in $G/H$ such that each vertex in $T$ is odd.
Deleting recursively all edges on a cycle when necessary, we obtain a $T$-join in $G/H$ that has no cycle. Replacing each vertex $v$ in such $T$-join by the corresponding contracted connected component of $H$, we obtain a union of  cycle-cubic trees of $G$, say $M$, relative to some connected components of $H$. We know that $M$ contains all $C_1, C_2, \ldots, C_{2m}$.
If $M$ does not contain all connected components of $H$, then adding the connected components of $H$ to $M$ until all connected components of $H$ are contained in $M$. This new graph is also a union of cycle-cubic trees, denoted by $M$. By Lemma \ref{c-c tree}, each connected component of $M$ has a function satisfying (1) and (2). The union of these functions, denoted by $f'$, also satisfies (1) and (2).
For each edge outside $M$, define $f(e)=-3$ and let $f(e)=f'(e)$ if $e\in M$. One can easily verify that $f$ is a zero-sum 6-flow of $G$.\qed
\end{proof}

\section{Acknowledgements}

Author appreciates Sanming Zhou for his valuable comments and advisements. Moreover, author acknowledges the hospitality of School of Mathematics and Statistics, The University of Melbourne during her visit while the research on this work was conducted.

\end{document}